\documentclass[10pt]{article}
\usepackage{graphicx,caption}
\usepackage{latexsym,booktabs}
\usepackage{blindtext,textcomp}
\usepackage{amsthm, amsmath, amssymb, amsbsy}
\usepackage{hyperref}
\usepackage{tikz}
\usepackage[utf8]{inputenc}
\usepackage{sectsty}
\sectionfont{\centering}
\usepackage{enumitem,mathrsfs}
\usepackage{comment}
\usepackage[colorinlistoftodos]{todonotes}

\newtheorem{theorem}{Theorem}[section]
\newtheorem{remark}[theorem]{Remark}
\newtheorem{cor}[theorem]{Corollary}

\newtheorem{lem}[theorem]{Lemma}
\newtheorem{defn}[theorem]{Definition}

\def \N {{\mathbb N}}

\def \R {{\mathbb R}}

  \author{Veer Singh Panwar\thanks{Email: vs728@snu.edu.in} \,\,and A. Satyanarayana Reddy\thanks{Email: satya.a@snu.edu.in}\\
Department of Mathematics\\Shiv Nadar University, Dadri\\ U.P. 201314, India.}
  \date{}
  
\begin{document}
	\title{\textbf{Positivity of Hadamard powers of a few band matrices}}
		\maketitle
 \begin{center}
 \large{\textbf{Abstract}}
 \end{center}
	\begin{footnotesize}
   Let $\mathbb{P}_G([0,\infty))$ and $\mathbb{P}_G^{'}([0,\infty))$ be the sets of positive semidefinite and positive definite matrices of order $n$, respectively, with nonnegative entries, where some positions of zero entries are restricted by a simple graph $G$ with $n$ vertices. It is proved that for a connected simple graph $G$ of order $n\geq 3$, the set of powers preserving positive semidefiniteness on $\mathbb{P}_G([0,\infty))$ is precisely the same as the set of powers preserving positive definiteness on $\mathbb{P}_G^{'}([0,\infty))$. In particular, this provides an explicit combinatorial description of the critical exponent for positive definiteness, for all chordal graphs.
   Using chain sequences, it is proved that the Hadamard powers preserving the positive (semi) definiteness of every tridiagonal matrix with nonnegative entries are precisely $r\geq 1$. The infinite divisibility of tridiagonal matrices is studied. The same results are proved for a special family of pentadiagonal matrices.
	\end{footnotesize}

\vspace{0.5 cm}		
	
	\begin{footnotesize}
	\end{footnotesize}
		
\begin{center}
\section{Introduction}\label{sec:intro}    
\end{center}

Throughout this paper, every matrix has real entries. A matrix is called \textit{nonnegative} if all its entries are nonnegative.  A matrix $A$ is called
{\em positive  semidefinite} (PSD) (respectively {\em positive definite} (PD)) if $A$ is symmetric and $\langle x,Ax \rangle \geq 0$ for all $x \in \R^n$ (respectively $\langle x,Ax \rangle > 0$ for all $x \in \R^n\setminus\{0\}$). If $r > 0$, then we denote the $r$th \textit{Hadamard power} of a nonnegative matrix $A = [a_{ij}]$ by $A^{or}$ (or $(A)^{\circ r})$, where $A^{or} = [a_{ij}^r]$. A lot of interest has been shown in studying the real entrywise powers preserving the positive semidefiniteness of various families of matrices, see \cite{bhatia2007positivity,fitzgerald1977fractional,guillot2015complete,guillot2016critical,hiai2009monotonicity,horn1969theory,jain2017hadamard,jain2020hadamard}. A well-known result is that if $A$ is a nonnegative PSD matrix of order $n$ and $r \geq n-2$, then $A^{\circ r}$ is PSD. Moreover, for every positive noninteger $r<n-2,$ there exists a positive semidefinite matrix $A$ such that $A^{\circ r}$ is not positive semidefinite (see \cite[Theorem 2.2]{fitzgerald1977fractional}). \par

Let $\mathcal{I} \subseteq \R$. A function $f$ defined on $\mathcal{I}$ is called \textit{superadditive} on $\mathcal{I}$ if $f(a+b) \geq f(a)+f(b)$ for all $a,b \in \mathcal{I}$. Let $G=(V,E)$ be a simple graph with vertex set $V=\{1,\ldots, n\}$ such that $n\geq 3$. Let $\mathbb{P}_n(\mathcal{I})$ and $\mathbb{P}_n^{'}(\mathcal{I})$, respectively, be the sets of all positive semidefinite and positive definite matrices of order $n$ with entries in $\mathcal{I}$. Let 
$$\mathbb{P}_G(\mathcal{I}) = \{A=[a_{ij}] \in \mathbb{P}_n(\mathcal{I}): a_{ij}=0 \,\, \text{for all}\,\, i\neq j, (i,j) \notin E \},$$
$$\mathbb{P}_G^{'}(\mathcal{I}) = \{A=[a_{ij}] \in \mathbb{P}_n^{'}(\mathcal{I}): a_{ij}=0 \,\, \text{for all}\,\, i\neq j, (i,j) \notin E \},$$
$$\mathcal{H}_G = \{r \in \R: A^{\circ r} \in \mathbb{P}_G(\R) \,\,\textup{for all}\,\, A \in \mathbb{P}_G([0,\infty))\},$$
$$\mathcal{H}_G^{'} = \{r \in \R: A^{\circ r} \in \mathbb{P}_G^{'}(\R) \,\,\textup{for all}\,\, A \in \mathbb{P}_G^{'}([0,\infty))\}.$$
Let $H$ be an induced subgraph of the graph $G$. Then $\mathcal{H}_G \subseteq \mathcal{H}_H$. \par
Our first result is as follows:

\begin{theorem} \label{thm: the two sets}
Let $G$ be any connected simple graph with at least $3$ vertices. Then $\mathcal{H}_G = \mathcal{H}_G^{'}$.
\end{theorem}

A matrix $A=[a_{ij}]$ is called a \emph{band matrix} of bandwidth $d$ if $a_{ij}=0$ for $|i-j|>d$. The band matrices of bandwidth $1$ (respectively $2$) are also called \emph{tridiagonal} (respectively \emph{pentadiagonal}). 
%if $a_{ij}=0$ for $|i-j|>1$ (respectively $a_{ij} = 0$ for $|i-j|>2$). 
Let the symmetric nonnegative band matrices $T$ and $P$ be defined as follows:

\begin{align} \label{eqn: T and P}
T= \begin{bmatrix}
a_{1}&b_{1}&&&&\\
b_{1}&a_{2}&b_{2}&&&\\
&b_{2}&\ddots&\ddots&&\\
&&\ddots&\ddots&\ddots&\\
&&&\ddots&\ddots&b_{n-1}\\
&&&&b_{n-1}&a_{n}
\end{bmatrix} \,\, \text{and}
\,\, P = \begin{bmatrix}
x_{1}& 0 &y_{1}& & & & & \\
0 &x_{2}& 0 &y_{2}&& &  &\\
y_{1}& 0 &\ddots&\ddots&\ddots& &\\
&y_2&\ddots&\ddots&\ddots&\ddots&\\
&&\ddots&\ddots&\ddots&\ddots&\ddots\\
&&&\ddots&\ddots&\ddots&\ddots&y_{n-2}\\
&&&&\ddots&\ddots&\ddots& 0 \\
& & & & &y_{n-2}& 0 &x_{n}
\end{bmatrix},
\end{align}

\noindent where $n\geq 3$, $a_{i}$, $b_{j}$, $x_i$ and $y_k \geq 0$ for $1 \leq i \leq n$, $1 \leq j \leq (n-1)$, $1 \leq k \leq (n-2)$. \par

A sequence $\{\mathfrak{a}_{k}\}_{k>0}$ is called a \textit{chain sequence} if there exists a parameter sequence $\{\mathfrak{g}_{k}\}_{k \geq 0}$ such that $0 \leq \mathfrak{g}_{0}<1$ and $0 < \mathfrak{g}_{k}<1$ for $k\geq1$ and $\mathfrak{a}_{k} = (1-\mathfrak{g}_{k-1})\mathfrak{g}_{k}$ for $k \geq 1$ (see \cite[p. 91]{chihara2011introduction}). A basic example of a chain sequence is the constant sequence $\big\{\frac{1}{4}\big\}_{k \geq 1}$ with the parameter sequence $\big\{\frac{1}{2}\big\}_{k \geq 1}$. For more information and examples on chain sequences, see \cite{chihara2011introduction,ismail1991discrete,wall2018analytic}. \par
A graph $G =(V,E),$ where $V = \{1, \ldots, n\},$ is called a \textit{band graph} of bandwidth $d$ if $\{i,j\}\in E$ if and only if $i \neq j$ and $|i-j|\leq d$.  Let $G$ be a band graph of bandwidth $1$ with $n\geq 3$, i.e., a path graph. Then $\mathbb{P}_G([0,\infty))$ is precisely the set of all PSD nonnegative tridiagonal matrices of order $n$. By Theorem $1.4$ in \cite{guillot2016critical}, $\mathcal{H}_{G}=[1,\infty)$. Hence, the Hadamard powers preserving the positive semidefiniteness of all the nonnegative tridiagonal matrices of order $n\geq 3$ are precisely $r\geq 1$. %We say that a matrix $A = [a_{ij}] \in M_n(\R)$ corresponds to a graph $G =(V,E)$ if $a_{ij}=0$ for $i\neq j$ and $(i,j)\not\in E$. \textcolor{red}{It's easy to observe that every band matrix of bandwidth $d$ corresponds to a band graph of bandwidth $d$}.
%Let $G = (V,E)$ be a band graph of bandwidth $d,$ where $V = \{1, \ldots, n\};$ and $A =[a_{ij}] \in M_{n}(\R)$ be a PSD matrix such that $a_{ij}=0$ for $i\neq j$ and $(i,j)\not\in E$. Then $A^{\circ r}$ is PSD for $r \geq min(d,n-2)$ \cite[Cor 3.11]{guillot2016critical}. Hence, if $T$ and $P$ are PSD matrices, which are tridiagonal and pentadiagonal, respectively, then $T^{\circ r}$ and $P^{\circ s}$ are PSD for $r\geq 1$ and $s\geq 2,$ respectively. 
We give an alternative proof for this in our next theorem using chain sequences.

%Thus by Theorem~\ref{thm: the two sets}, the same statement holds for the positive definiteness. %$T^{\circ r}$ is PD for $r>1$ if $T$ is a PD matrix. We prove it alternatively in our next theorem using chain sequences.

\begin{theorem} \label{Thm : sub main result}
%Let the matrix $T$ in (\ref{eqn: T and P}) be \textup{PD}. Then $T^{\circ r}$ is \textup{PD} for $r>1$.
The matrix $T^{\circ r}$ is \textup{PD (PSD)} for every \textup{PD (PSD)} matrix $T$ in (\ref{eqn: T and P}) if and only if $r\geq1$.
\end{theorem}

Similarly, for $n\geq 3$, let $G$ be a graph with vertex set $V=\{1,\ldots, n\}$, which is the disjoint union of two path graphs with vertex sets $V_1=\{1,3,\ldots, p\}$ and $V_2=\{2,4,\ldots, q\}$, where $p = n-1$, $q=n$ if $n$ is even, and $p=n$, $q= n-1$ if $n$ is odd. Then $\mathbb{P}_G([0,\infty))$ is precisely the set of all PSD nonnegative pentadiagonal matrices as in Equation (\ref{eqn: T and P}). Hence, by Theorem~$1.4$ of \cite{guillot2016critical}, $\mathcal{H}_G=[0,\infty)$ for $n=3,4$, and $\mathcal{H}_G=[1,\infty)$ for $n\geq 5$. We prove the latter result alternatively in our next Theorem.

\begin{theorem} \label{Penta results}
%Let $P = [p_{ij}]$ be a pentadiagonal matrix of order $n$ as given in Equation \eqref{eqn: T and P}, where $p_{ij} = 0$ for $|i-j|=1$, and $p_{ij}>0$ for $|i-j|=0$ and $2$. 
 The matrix $P^{\circ r}$ is \textup{PD (PSD)} for every \textup{PD (PSD)} matrix $P$ in (\ref{eqn: T and P}) of order $n\geq 5$ if and only if $r\geq1$.
\end{theorem}

A nonnegative symmetric matrix $A$ is said to be {\em infinitely divisible} (ID) if $A^{or}$ is PSD for every $r > 0$. It is obvious that every ID matrix is PSD; however, the converse need not be true (see \cite{bhatia2006infinitely}). Some basic examples of ID matrices are nonnegative PSD matrices of order $2$ and diagonal matrices with nonnegative diagonal entries. For more examples and results on ID matrices, see~\cite{bhatia2006infinitely,bhatia2007mean,grover2020positivity,horn1969theory}. In our next theorem, we give a characterization for the matrix $T$ to be infinitely divisible.

\begin{theorem}{\label{TID:thm1}}
   The matrix $T$ in (\ref{eqn: T and P}) is \textup{ID} if and only if T is \textup{PSD} and $b_ib_{i+1}=0$ for every $i\in \{1,2,\ldots,n-2\}$. %{\it i.e.,} the sequence $\{b_{i}\}_{i\geq 1}$ has no two consecutive positive entries. 
\end{theorem}

 In Section~\ref{sec: Main results}, we give proofs of the above results, concluding with some related remarks. 
 
 \section{Proofs of the results}\label{sec: Main results}

Let $P_n$ and $K_n$ denote the path graph and the complete graph on $n$ vertices, respectively. Every connected graph $G$ with at least $3$ vertices contains at least a path graph $H= P_3$ or a triangle $H = K_3$ as an induced subgraph. By Theorem $1.4$ in \cite{guillot2016critical}, $\mathcal{H}_{P_3} = \mathcal{H}_{K_3} = [1,\infty)$. Hence, $\mathcal{H}_G \subseteq [1,\infty)$. Let $a,b \geq 0,(a,b)\neq (0,0)$ and $r\geq 1$. Then $$(a+b)^r=\left[\frac{a}{a+b}(a+b)^r+\frac{b}{a+b}(a+b)^r\right]\geq \left[\frac{a}{a+b}(a+b)\right]^r+\left[\frac{b}{a+b}(a+b)\right]^r= a^r+b^r.$$ Hence, the function $f(x)=x^r$ is superadditive on $[0,\infty)$ for $r\geq 1$. We now prove our first result. \\

\hspace{-0.55 cm}\textit{\textbf{Proof of Theorem~\ref{thm: the two sets}}}.
Let $r \in \mathcal{H}_G^{'}$ and $A \in \mathbb{P}_G([0,\infty))$. Let $I$ denote the identity matrix of order $n$. Since $A$ is PSD, there exists a sequence $\{A_{k}\}_{k\geq 1}$ of PD matrices, where the matrices $A_{k} = A+\frac{1}{k}I \in \mathbb{P}_{G}^{'}([0,\infty))$ converges to $A$ entrywise as $k \rightarrow \infty$. Hence, the matrices $A_{k}^{\circ r}$ are PD for $k\geq 1$, so their limit $A^{\circ r}$ is PSD. Hence, $r \in \mathcal{H}_G$. \par
 Conversely, let $r \in \mathcal{H}_G,$ then $r \geq 1$. Let $A \in \mathbb{P}_G^{'}([0,\infty))$. Let $B= A-\lambda I$, where $\lambda>0$ is the smallest eigenvalue of $A$. Then $B \in \mathbb{P}_G([0,\infty)),$ so $B^{\circ r}$ is PSD. 
 
 $$((B+\lambda I)^{\circ r}-B^{\circ r})_{ij} = \begin{cases}
 (b_{ii}+\lambda)^r-b_{ii}^r & \text{if}\,\, i=j,\\
 0 & \text{otherwise}.
 \end{cases}$$ 
 Since the function $f(x) = x^r$ is superadditive on $[0,\infty)$, we have $(b_{ii}+\lambda)^r-b_{ii}^r\geq\lambda^{r}>0$. Hence, the diagonal matrix $(B+\lambda I)^{\circ r}-B^{\circ r}$ is PD. Therefore, $A^{\circ r} = ((B+\lambda I)^{\circ r}-B^{\circ r})+ B^{\circ r}$ is PD, which gives $r \in \mathcal{H}_G^{'}$. This completes the proof. \qed

\begin{defn}
A graph $G$ is called a \textit{chordal graph} if every cycle of four or more vertices in it has a chord.
\end{defn}

Let $K_n^{(1)}$ denote the complete graph on $n$ vertices with one edge missing. By Theorem $1.4$ in \cite{guillot2016critical} and Theorem~\ref{thm: the two sets}, we have a combinatorial characterization of the critical exponent for any
chordal graph:  

\begin{cor}
Let $G$ be any chordal graph with at least $3$ vertices and $r$ be the largest integer such that either $K_r$ or $K_r^{(1)}$ is a subgraph of $G$. Then $\mathcal{H}_G^{'} = \N \cup [r-2,\infty)$.
\end{cor}

We now return to the analysis of powers preserving the positivity of the matrices $T$ and $P$, defined in Equation~(\ref{eqn: T and P}). Let $b_j=0$ for some $1 \leq j \leq (n-1)$. Then $T$ becomes a block diagonal matrix having two smaller diagonal blocks. Continuing this way with these smaller blocks and repeating the process, one can see that $T$ is a block diagonal matrix, where each diagonal block is a tridiagonal matrix with positive entries on its upper and lower diagonals. Moreover, every such block of $T$ is a PD matrix with positive entries on the main, upper and lower diagonals, if $T$ is PD. \par

%Our next result is a direct consequence of Theorem \ref{thm: the two sets} and the fact that $[1,\infty)\subseteq \mathcal{H}_G$ for any band graph of bandwidth $1$. We give an alternate proof for it using chain sequences.\par

To prove our next result, we will need the following theorems related to chain sequences.
\vspace{0.2cm}

\begin{theorem}\textup{\cite[Theorem 5.7]{chihara2011introduction}} \label{thm : comparison test}
 If $\{\mathfrak{a}_{k}\}_{k=1}^{n}$ is a chain sequence and $0< \mathfrak{c}_{k}\leq \mathfrak{a}_{k}$ for $k \geq 1,$ then $\{\mathfrak{c}_{k}\}_{k=1}^{n}$ is also a chain sequence.
\end{theorem}

\begin{theorem}\textup{\cite[Theorem 3.2]{ismail1991discrete}}
 \label{thm : Wall and Wetzel}
 Let $a_i,b_j>0$ for $1 \leq i \leq n, 1\leq j \leq (n-1)$. Then $T$ is positive definite if and only if $\Big\{\frac{b_{j}^{2}}{a_{j}a_{j+1}}\Big\}_{j=1}^{n-1}$ is a chain sequence. \end{theorem}

We now prove our second result.\\

\hspace{-0.55cm}\textit{\textbf{Proof of Theorem~\ref{Thm : sub main result}}}. By Theorem~\ref{thm: the two sets}, giving the proof for the \textup{PD} case is sufficient. We first prove the `if part'. It is enough to prove our result for the matrix $T$, where $a_i,b_j>0$ for $1 \leq i \leq n, 1\leq j \leq (n-1)$.
By Theorem~\ref{thm : Wall and Wetzel}, $\Big\{\frac{b_{j}^{2}}{a_{j}a_{j+1}}\Big\}_{j=1}^{n-1}$ is a chain sequence. Let $r>1$. Since $T$ is PD, $0<\frac{b_{j}^{2}}{a_{j}a_{j+1}}< 1$, which gives $0 < \left(\frac{b_{j}^{2}}{a_{j}a_{j+1}}\right)^{r}< \frac{b_{j}^{2}}{a_{j}a_{j+1}}$. Thus by Theorem~\ref{thm : comparison test}, $\Big\{\frac{b_{j}^{2r}}{a_{j}^{r}a_{j+1}^{r}}\Big\}_{j=1}^{n-1}$ is also a chain sequence. Hence, by Theorem~\ref{thm : Wall and Wetzel}, the matrix $T^{\circ r}$ is PD. \\
Now we prove the `only if' part. Let $0<r<1$. Consider  the tridiagonal \textup{PD} matrix $$A(\epsilon) = \begin{bmatrix}
 1 & 1 & 0\\
 1 & (2+\epsilon) & 1\\
 0 & 1 & 1
\end{bmatrix},$$ where $\epsilon$ is any arbitrary positive number. For every $0<\epsilon < (2^{\frac{1}{r}}-2),$  $\det(A(\epsilon)^{\circ r}) = (2+\epsilon)^{r}-2 <0,$ so $A(\epsilon)^{\circ r}$ is not \textup{PD}. Hence we are done.\qed \\

A symmetric block diagonal matrix is PSD (PD) if and only if each block is PSD (PD). Let $A$ be any matrix of order $n$. For $\alpha = \{\alpha_{1}, \ldots, \alpha_{k}\} \subseteq \{1,2,\ldots, n\}$, where $\alpha_{1} < \alpha_{2} < \cdots < \alpha_{k}$, let $A[\alpha]$ denote the principal submatrix of $A$ obtained by picking rows and columns indexed by $\alpha$. $A$ is PD if and only if all its leading principal minors are positive. If $A$ is PD, then all its principal submatrices are PD. For distinct positive integers $1 \leq i_1,\ldots,i_n \leq n$, let  $\textup{Perm}(i_1,\ldots,i_n)$  denote the permutation matrix of order $n$, whose $k$th row is the $i_{k}$th row of the identity matrix of order $n$. If $A$ is a PD (PSD) matrix of order $n$, then $XAX^{*}$ is PD (PSD) for any nonsingular matrix $X$ of order $n$. \par 
Our third result is as given below.\\

\hspace{-0.55cm}\textit{\textbf{Proof of Theorem~\ref{Penta results}}}. 
 We first show the ‘if part’. Let $A_{l}^{*} = P[\beta]$ and $A_{m}^{**} = P[\gamma],$ where $\beta = \{1,3,\ldots, (2l-1)\}, \gamma = \{2,4,\ldots, 2m\}$ for $1 \leq l,m \leq k$ if $n=2k$ and $1\leq l \leq (k+1), 1 \leq m \leq k$ if $n=2k+1$. One can observe that the principal submatrices $A_{l}^{*}$ and $A_{m}^{**}$ of $P$ are tridiagonal with the upper and lower diagonal entries belonging to the set $\{y_{i}\}_{i=1}^{(n-2)},$ and the main diagonal entries belonging to the set $\{x_{i}\}_{i=1}^{n}$ as given below:

\begin{align*}
A_{l}^{*} = \begin{bmatrix}
x_{1} & y_{1} & \cdots & 0 & 0\\
y_{1} & x_{3} & y_3 & 0 & 0\\
\vdots & y_3 & \ddots & \ddots & \vdots\\
0 & 0 & \ddots & x_{2l-3} & y_{2l-3}\\
0 & 0 & \cdots & y_{2l-3} & x_{2l-1}
\end{bmatrix}_{l\times l}
\text{and}\,\,\,
A_{m}^{**} = \begin{bmatrix}
x_{2} & y_{2} & \cdots & 0 & 0\\
y_{2} & x_{4} & y_4 & 0 & 0\\
\vdots & y_4 & \ddots & \ddots & \vdots\\
0 & 0 & \ddots & x_{2m-2} & y_{2m-2}\\
0 & 0 & \cdots & y_{2m-2} & x_{2m}
\end{bmatrix}_{m\times m}.
\end{align*}

Also note that for every $r>0,$ $P^{\circ r}$ is congruent to the block matrix $M^{\circ r}$ via a permutation matrix $X$ of order $n$, {\it i.e.,} $M^{\circ r} = XP^{\circ r}X^{*}$ for $r>0$, where 
\begin{align} \label{eqn: proof of penta matrix}
X = 
\begin{cases}
\textup{Perm}(1,3,\ldots,(2k-1),2,4,\ldots,2k) & \textup{if}~ n = 2k, \\
\textup{Perm}(1,3,\ldots,(2k+1),2,4,\ldots,2k) & \textup{if}~ n = 2k+1
\end{cases} \,\,\text{and}\,\,
M = \begin{cases}
\begin{bmatrix}
 A_{k}^{*} & 0\\
 0 & A_{k}^{**}
\end{bmatrix} & \textup{if}~ n = 2k,\\\\
\begin{bmatrix}
 A_{k+1}^{*} & 0\\
 0 & A_{k}^{**}
\end{bmatrix} & \textup{if}~ n = 2k+1.
\end{cases}
\end{align}

We prove the required results for the case when $n$ is even (the case when $n$ is odd can be proved analogously). Let $r>1$ and $n = 2k$. If $P$ is PD (PSD), then because $M = XPX^{*}$, $M$ is PD (PSD). So $A_{k}^{*}$ and $A_{k}^{**}$ are PD (PSD) matrices. Hence, $(A_{k}^{*})^{\circ r}$ and $(A_{k}^{**})^{\circ r}$ are PD (PSD), which gives $M^{\circ r}$ is PD (PSD). But then $P^{\circ r} = X^{-1}M^{\circ r}X$ is PD (PSD).\par
For the `only if' part of the PSD case, the following example is sufficient. Let
\begin{align*}
\mathcal{P}=\begin{bmatrix}
  1&0&1&0&0\\
  0&2&0&1&0\\
  1&0&2&0&1\\
  0&1&0&1&0\\
  0&0&1&0&1
   \end{bmatrix}.
\end{align*}
The matrix $\mathcal{P}$ is PSD, but $\det(\mathcal{P}^{\circ r})= 2 - 3(2^r) + 4^r<0$ for any $0<r<1$. Hence, $\mathcal{P}^{\circ r}$ is not PSD for any $0<r<1$. Since $\mathcal{P}$ is the limit of a sequence of PD  pentadiagonal matrices (in the form given in Equation~(\ref{eqn: T and P})), the `only if' part of the PD case is also done. This completes the proof.
\qed \\

Each principal submatrix of $A$ is ID if $A$ is ID. Every PSD matrix of order $2$ is ID. We now discuss the infinite divisibility of the matrices $T$ and $P$. 

\begin{lem} \label{lem:3TID}
Let $A =\begin{bmatrix}
 a_{1}&b_{1}&0\\
b_{1}&a_{2}&b_{2}\\
0&b_{2}&a_{3}
\end{bmatrix}$ be a \textup{PSD} matrix of order $3$. Then $A$ is \textup{ID} if and only if $b_{1}b_{2}=0$.
\end{lem}

\begin{proof}
 Let $A$ be ID and $C = \lim_{r \rightarrow 0^{+}}A^{\circ r}$. Since the matrix $C$ is the limit of a sequence of PSD matrices, it is PSD. Let $b_1$ and $b_2$ be positive. Since $A$ is PSD, $a_1,a_2$ and $a_3$ are positive. This gives that $C = \begin{bmatrix}
 1 & 1 & 0\\
 1 & 1 & 1\\
 0 & 1 & 1 
\end{bmatrix}$. Hence, $C$ is not PSD, which is not true. So $b_1b_2=0$. Conversely, if $b_1b_2 = 0,$ then $\det(A^{\circ r}) \geq 0$ for $r>0$. Hence, $A$ is ID.
\end{proof}
In our final result, we give a characterization for the matrix $T$ in Equation (\ref{eqn: T and P}) to be infinitely divisible.\\

\hspace{-0.55 cm}\textit{\textbf{Proof of Theorem~\ref{TID:thm1}}}.
 Let $T$ be ID. Let $b_{k}b_{k+1}>0,$ for some $1\leq k \leq (n-2),$ then by Lemma~\ref{lem:3TID}, the principal submatrix $$\begin{bmatrix}
  a_k & b_k & 0\\ 
  b_{k} & a_{k+1} & b_{k+1}\\ 
  0 & b_{k+1} & a_{k+2}                                 \end{bmatrix}$$ of $T$ is not ID. Thus, $T$ is not ID, which contradicts the hypothesis. Hence, $b_ib_{i+1}=0$ for every $i\in \{1,2,\ldots,n-2\}$. Conversely, if $b_ib_{i+1}=0$ for every $i\in \{1,2,\ldots,n-2\},$ then $A$ becomes a block diagonal matrix, where each non-zero diagonal block is a PSD matrix of order $1$ or $2$. Hence, $T$ is ID.  This completes the proof. \qed
\begin{cor}
  The matrix $P$ in (\ref{eqn: T and P}) is \textup{ID} if and only if $P$ is \textup{PSD} and the sequences $\{y_{2i}\}_{i \geq 1}$ and $\{y_{2i-1}\}_{i\geq 1}$ have no two consecutive positive entries. Hence, for $n=3$ and $4$, the matrix $P$ is \textup{ID} if and only if $P$ is \textup{PSD}.
\end{cor}
  
  \begin{proof}
    From Equation~(\ref{eqn: proof of penta matrix}), we have $M^{\circ r} = XP^{\circ r}X^{*}$ for every $r>0$. So $P$ is ID if and only if $M$ is ID. Hence the result holds by Theorem~\ref{TID:thm1}.
  \end{proof}
 
We end with a few related remarks.
  
%\begin{remark}The result analogous to Theorem~\ref{Thm : sub main result} is not true for all band matrices. For example, let \begin{align*}K = \begin{bmatrix}2 & 3 & 4 & 5 & 0\\3 & 5 & 7 & 9 & 0\\
     %4 & 7 & 10 & 13 & 0\\
     %5 & 9 & 13 & 17 & 0\\
     %0 & 0 & 0 & 0 & 1
    %\end{bmatrix},
%\end{align*}
%then $K$ is a band matrix of bandwidth $3$, which is \textup{PSD}, but $K^{\circ r}$ is not \textup{PSD} for $1<r<2$. Since $K$ is the limit of a sequence of \textup{PD} band matrices $\{K_k\}_{k\geq 1}$ of bandwidth $3$, where the matrices $K_k=K+\frac{1}{k}I$, the result analogous to Theorem~\ref{Thm : sub main result} doesn't hold for all band matrices of bandwidth $3$.
%\end{remark}

\begin{remark}\label{rem: block diagonal representation}
   From Theorem~\ref{TID:thm1}, the matrix $T$ is \textup{ID} if and only if $T$ is a block diagonal matrix, where each non-zero diagonal block is a \textup{PSD} matrix of order $1$ or $2$.
  \end{remark}
  
  In general, ID matrices are not closed under addition and multiplication. For example,  let $X = [x_{i}x_{j}],$ where $x_{1}, \ldots, x_{n}$ are distinct positive real numbers and $J_{n}$ be the matrix of order $n$ with each of its entries equals to $1,$ then $X$ and $J_{n}$ are both ID, but their sum is not ID (see \cite[Theorem $1.1$]{jain2017hadamard}). The Cauchy matrix $\mathcal{C} = [c_{ij}] =  \left[\frac{1}{i+j}\right], 1 \leq i,j \leq 3$ is ID (see \cite{bhatia2006infinitely}), but its square $\mathcal{C}^{2}$ is not ID because $\det(\mathcal{C}^{2})^{\circ \frac{1}{4}} = \det \left[\big(\frac{1}{i+j}\big)^{\frac{1}{4}}\right] <0$. We say that two block diagonal matrices are of the same structure if their corresponding blocks are square matrices of the same order. The set of block diagonal matrices of the same structure is closed under addition, multiplication and multiplication by a nonnegative scalar. Hence, by Remark~\ref{rem: block diagonal representation}, we get the following.

\begin{remark}
Let $m \geq 1$ 
%and $f(x) = \sum_{k=0}^{m} a_k x^{k}$ be a polynomial, where 
and $a_k \geq 0$ for all $0\leq k \leq m$. Let $T^{0}=T^{\circ 0}=I$. If the tridiagonal matrix $T$ (as in Equation~(\ref{eqn: T and P})) is \textup{ID}, then the matrices $f(T) =\sum_{k=0}^{m} a_k T^{k}$ and $f[T]=\sum_{k=0}^{m} a_k T^{\circ k}$ are \textup{ID}.
\end{remark}

\textbf{Acknowledgement:} \textit{It is a pleasure to express our heartfelt gratitude to Apoorva Khare for his insightful remarks that helped to enhance the article. We are also appreciative to the anonymous referees for their thorough review and valuable remarks.}


\begin{thebibliography}{10}
\scriptsize

\bibitem{bhatia2006infinitely} 
R. Bhatia. Infinitely divisible matrices. {\em Am. Math. Mon.}, 113(3):221--235, 2006.

\bibitem{bhatia2007positivity}
R. Bhatia and L. Elsner. Positivity preserving Hadamard matrix functions. \textit{Positivity}, 11(4):583--588, 2007.

\bibitem{bhatia2007mean}
R. Bhatia and H. Kosaki. Mean matrices and infinite divisibility. \textit{Linear Algebra Appl.}, 424(1):36--54, 2007.

\bibitem{chihara2011introduction}
T. S. Chihara. \textit{An Introduction to Orthogonal Polynomials}. Courier Corporation, Massachusetts, United States, 2011.

\bibitem{fitzgerald1977fractional}
C. H. FitzGerald and R. A. Horn. On fractional Hadamard powers of positive definite matrices. \textit{J. Math. Anal. Appl.}, 61(3):633--642, 1977.

\bibitem{grover2020positivity}
P. Grover, V. S. Panwar, and A. S. Reddy. Positivity of some special matrices. \textit{Linear Algebra Appl.}, 596:203--215, 2020.

\bibitem{guillot2015complete}
D. Guillot, A. Khare, and B. Rajaratnam. Complete characterization of Hadamard powers preserving Loewner positivity, monotonicity, and convexity. \textit{J. Math. Anal. Appl.}, 425(1):489--507, 2015.

\bibitem{guillot2016critical}
D. Guillot, A. Khare, and B. Rajaratnam. Critical exponents of graphs. \textit{J. Comb. Theory Ser. A.}, Series A, 139:30--58, 2016.

\bibitem{hiai2009monotonicity}
F. Hiai. Monotonicity for entrywise functions of matrices. \textit{Linear Algebra Appl.}, 431(8):1125--1146, 2009.

\bibitem{horn1969theory}
R. A. Horn. The theory of infinitely divisible matrices and kernels. {\em Trans. Am. Math. Soc.}, 136:269--286, 1969.

\bibitem{ismail1991discrete}
M. E. H. Ismail and M. E. Muldoon. A discrete approach to monotonicity of zeros of orthogonal polynomials. {\em Trans. Am. Math. Soc.}, 323(1):65--78, 1991.

\bibitem{jain2017hadamard}
T. Jain. Hadamard powers of some positive matrices. \textit{Linear Algebra Appl.}, 528:147--158, 2017.

\bibitem{jain2020hadamard}
T. Jain. Hadamard powers of rank two, doubly nonnegative matrices. \textit{Adv. Oper. Theory}. 5(3):839--849, 2020.

\bibitem{wall2018analytic}
H. S. Wall. \textit{Analytic Theory of Continued Fractions}. Courier Dover Publications, New York, United States, 2018.
  

\end{thebibliography}
\end{document}